\documentclass[reqno, 11pt]{amsart}

\usepackage[utf8]{inputenc}
\usepackage[T1]{fontenc}
\usepackage{graphicx, enumerate, enumitem, amsmath, amssymb,amsthm,manfnt,hyperref,braket,amscd,
mathrsfs, mathtools, empheq, verbatim, placeins, floatrow, dsfont,color,soul, csquotes, tikz-cd, array, bbm}

\numberwithin{equation}{section}
\newtheorem{theorem}{Theorem}[section]
\newtheorem{proposition}{Proposition}[section]
\newtheorem{lemma}{Lemma}[section]

\newcommand{\fc}{\mathcal{C}}
\newcommand{\D}{\mathbb{D}}
\newcommand{\fp}{\mathcal{P}}

\theoremstyle{definition}

\newcommand{\C}{\mathbb{C}}
\newcommand{\N}{\mathbb{N}}
\newcommand{\Z}{\mathbb{Z}}
\newcommand{\fo}{\mathcal{O}}
\newcommand{\A}{\mathcal{A}}

\newcommand{\Smooth}{\A^{\infty}(\Omega)}
\newcommand{\poly}{\A^{\infty}(\fp)}

\newcommand\norm[1]{\left\lVert#1\right\rVert}
\newcommand{\abs}[1]{\left\vert#1\right\vert}
\newcommand{\normp}[1]{{\left\vert\kern-0.25ex\left\vert\kern-0.25ex\left\vert #1 
    \right\vert\kern-0.25ex\right\vert\kern-0.25ex\right\vert}}
\newcommand{\floor}[1]{\left\lfloor #1 \right\rfloor}

\begin{document}

\title[Laurent series]{Laurent series of holomorphic functions smooth up to the boundary}

\author{Anirban Dawn}
\address[Anirban Dawn]{Department of Mathematics, Central Michigan University\\%
             1200 S Franklin St, Mt Pleasant, MI 48859, USA}%
\email{dawn1a@cmich.edu}
\begin{abstract}
 It is shown that the Laurent series of a holomorphic function smooth up to the boundary on a Reinhardt domain in $\C^n$ converges unconditionally to the function in the Fr\'{e}chet topology of the space of functions smooth up to the boundary.  \end{abstract}
\maketitle

\section{introduction}
Let $\Omega \subset \C^n$ be a domain i.e. $\Omega$ is an open and connected subset of $\C^n$. Denote by $\Smooth$ the space of holomorphic functions smooth up to the boundary of $\Omega$, i.e. the space of holomorphic functions whose derivatives of all orders can be extended continuously up to the boundary. For a sequence of functions $\{f_j\} \subset \Smooth$, $f_j \rightarrow f$ in $\Smooth$ means that for every compact subset $K \subset \overline{\Omega}$, $f_j \rightarrow f$ uniformly on $K$ along with all partial derivatives. In particular, if $\Omega$ is bounded, then $f_j \rightarrow f$ in $\Smooth$ means that $f_j \rightarrow f$ uniformly on $\overline{\Omega}$ along with all partial derivatives.\\
Recall that a domain $\Omega \subset \C^n$ is called a \emph{Reinhardt domain} if for $z = (z_1,\cdots,z_n) \in \Omega$, one has $(\lambda_1 z_1, \cdots, \lambda_n z_n) \in \Omega$, where $\abs{\lambda_j} = 1$ for $j = 1,2,\cdots,n$. For a detailed exposition of Reinhardt domains, see \cite{MR2396710}. Let $\Omega$ be a Reinhardt domain and $f \in \fo(\Omega)$, the space of holomorphic functions on $\Omega$. It is well known that $f$ admits a unique Laurent series expansion which converges absolutely and uniformly on compact subsets of $\Omega$ to the function $f$, i.e. the Laurent series of $f$ converges to $f$ in the Fr\'{e}chet topology of $\fo(\Omega)$ (see \cite[p. 46]{MR847923}). The focus of this paper is to prove a result similar to this, for the space $\Smooth$.
\begin{theorem}\label{theorem1}
Let $\Omega$ be a Reinhardt domain in $\C^n$ and $f \in \Smooth$. Then the Laurent series of $f$ converges unconditionally to the function $f$ in the topology of $\Smooth$.
\end{theorem}
We say a formal series $\sum_{\alpha \in \Gamma} x_{\alpha}$, where $\Gamma$ is a countable index set, in a locally convex topological vector space (LCTVS) $X$ is \emph{unconditionally convergent} if for every bijection $\sigma: \N \coloneqq \{0,1,2,\cdots\} \rightarrow \Gamma$, the series $\sum_{j=0}^{\infty} x_{\sigma(j)}$ converges in the topology of $X$ (see \cite[p. 9]{MR1442255}).
\\
Convergence results similar to Theorem \ref{theorem1} for other classical function spaces have been proved. For $1 < p < \infty$, it is well known that the partial sums of the Taylor series of $f$ in $H^p(\D)$, the Hardy space on the unit disc in $\C$, converges to $f$ in the $H^p(\D)$ norm (see \cite[p. 104-110]{MR628971}). For the same range of $p$, convergence of partial sums of the Taylor series of $f$ in $A^p(\D)$, the space of holomorphic $L^p$ functions on the unit disc, has been proved in \cite{MR1091512}. For $p=1$, the sequence of partial sums does not converge in $H^1(\D)$ and $A^1(\D)$ norms. However, it has been shown in \cite{MR3861907} that the sequence of partial sums of $f \in H^1(\D)$ is norm convergent in the weaker norm $A^1(\D)$. A more general result can be found in \cite{MR3873548} where it is proved that for a bounded Reinhardt domain $\mathcal{R}$ in $\C^n$, the \enquote{square partial sums} of the Laurent series of $f$ in $A^p(\mathcal{R})$ converges to the function $f$ in the $A^p(\mathcal{R})$ norm. Notice that for a general Reinhardt domain $\Omega$, the convergence of the Laurent series in $\Smooth$ and $\fo(\Omega)$ is unconditional, which is not the case in $H^p(\D)$ or $A^p(\mathcal{R})$.\\ 
Theorem \ref{theorem1} is interesting because of the intrinsic importance of the space $\Smooth$ in complex analysis. For example, it is known that each smoothly bounded pseudoconvex domain $\Omega$ is a so called $\Smooth$-domain of holomorphy (see \cite{MR628348} and \cite{MR578928} for details). However, it is also known that pseudoconvex domains with non-smooth boundaries may not be $\Smooth$-domain of holomorphy. This was first noticed for Hartogs triangle $\{z=(z_1, z_2) : \abs{z_1} < \abs{z_2} < 1\} \subset \C^2$ by Sibony (see \cite{MR385164}) and generalised recently by Chakrabarti to Reinhardt domains in $\C^n$ with $0$ as a boundary point (see \cite{MR3981123}).\\ 
The paper is organised as follows. In Section \ref{unconditionalconvergence} we define the notion of absolute convergence of a series in $\Smooth$ and we prove that absolute convergence of a series in $\Smooth$ implies unconditional convergence. In addition, we show that absolutely convergent series in $\Smooth$ converges in the net of partial sums (see Section \ref{unconditionalconvergence} for more details). At the end we prove Theorem \ref{theorem1} in Section \ref{mainresult} by showing that the Laurent series of $f \in \Smooth$ converges absolutely (and therefore unconditionally) in $\Smooth$ to the function $f$.   

\subsection{Acknowledgements}
The author would like to thank Debraj Chakrabarti for his valuable suggestions, encouragement and support to this work. The author is grateful to the Department of Mathematics, Central Michigan University for providing research assistantship during this work. Also, many thanks to the referee for valuable feedback and comments.
\section{Preliminaries}

\subsection{The topology of $\Smooth$}\label{gensem}
Let $\Omega \subset \C^n$ be a domain. We now describe the topology of the space $\Smooth$, the space of holomorphic functions smooth up to the boundary of $\Omega$. First, assume $\Omega$ is bounded. Then $\Smooth = \bigcap_{k \in \N} \A^k(\Omega)$, where for every $k \in \N$, $\A^k(\Omega) \coloneqq \fc^{k}(\overline{\Omega}) \bigcap \fo(\Omega)$ and $\fc^k(\overline{\Omega})$ denotes the space of $k$-times continuously differentiable $\C$-valued functions whose derivatives up to order $k$ can be extended continuously up to the boundary of $\Omega$. The space $\Smooth$ is a Fr\'{e}chet space and its Fr\'{e}chet topology is generated by the $\fc^k$-seminorms given by,
\begin{equation}\label{eq:1}
\norm{f}_{k,\Omega} \coloneqq \sup \bigg \{ \abs{D^{\alpha} f(z)} : z \in \Omega, [\alpha] \leq k \bigg \}    
\end{equation}
where $k$ ranges over $\N$, $\alpha = (\alpha_1, \cdots , \alpha_n) \in \N^n$ is a multi-index with length $[\alpha] = \sum_{j=1}^n \alpha_j$, and 
\begin{equation}\label{eq:2}
D^{\alpha} f =  \frac{\partial^{[\alpha]} f}{\partial z_1^{\alpha_1}...\partial z_n^{\alpha_n}}.    
\end{equation}
Let $(X, \tau)$ be an LCTVS. Recall that a collection of continuous seminorms $\{p_k : k \in \Lambda\}$ on $X$, where $\Lambda$ is an index set, \emph{generates} the topology $\tau$ if for every continuous seminorm $p$ on $X$, there exists a finite subset $F \subset \Lambda$ and a $C > 0$ such that 
\begin{equation}\label{eq:1B}
p(x) \leq C \cdot \max_{k \in F} \hspace{1mm} p_{k}(x) \hspace{3mm} \text{for all $x \in X$}.
\end{equation}
Now, assume $\Omega$ is unbounded. For $m \in \N$, let $\Omega_m = \Omega \bigcap P_m$ where $P_m = \{z: \abs{z_j} < m \hspace{1mm} \text{for all $j$}\}$ is the polydisc of radius $m$. Then $\Omega_m$ is bounded for each $m$ and we write $\Omega = \bigcup_{m = 0}^{\infty} \Omega_m$. The Fr\'{e}chet topology of $\Smooth$ is generated by the collection of seminorms $\{\norm{f}_{k, \Omega_m}: k, m \in \N\}$, where
\begin{equation}\label{eq:1A}
\norm{f}_{k,\Omega_m} \coloneqq \sup \bigg \{ \abs{D^{\alpha} f(z)} : z \in \Omega_m, [\alpha] \leq k \bigg \}.    
\end{equation} 
Note that for a sequence of functions $\{f_N\} \subset \Smooth$, $f_N \rightarrow f$ in $\Smooth$ as $N \rightarrow \infty$ if and only if $f_N \rightarrow f$ in $\A^{\infty}(\Omega_m)$ for every $m \in \N$, as $N \rightarrow \infty$.\\
Now we describe another collection of seminorms that generates the same locally convex topology of $\Smooth$, where $\Omega \subset \C^n$ is bounded. For $\alpha \in \Z^n$, we define
\begin{equation}\label{eq:22}
\abs{\alpha}_{\infty} \coloneqq \max \big \{ \abs{\alpha_j}, 1 \leq j \leq n \big \}.
\end{equation}
For $k \in \N$, define
\begin{equation}\label{eq:23}
 \widetilde{\A}^k(\Omega) \coloneqq \Big \{f \in \A^k(\Omega) : D^{\alpha}(f) \in \A^0(\Omega) \hspace{2mm}\text{where $\abs{\alpha}_{\infty} \leq k$} \Big \},   
\end{equation}
where $\alpha = (\alpha_1, \cdots, \alpha_n)$ is a multi-index in $\N^n$ and $D^{\alpha}$ is defined in \eqref{eq:2}. Note that  $\widetilde{\A}^k(\Omega)$ is a Banach space with the norm,
\begin{equation}\label{eq:24}
\normp{f}_{k, \Omega} = \sup  \bigg \{\abs{D^{\alpha} f(z)} : z \in \Omega , \abs{\alpha}_{\infty} \leq k \bigg \}.   \end{equation}
When $n=1$, $\widetilde{\A}^k(\Omega)$ coincides with $\A^k(\Omega)$. Observe that for $n \geq 2$, $\A^{nk}(\Omega) \subsetneq \widetilde{\A}^k(\Omega) \subsetneq \A^k(\Omega)$. Moreover, for each $k \in \N$, the inclusion maps $\A^{nk}(\Omega) \overset{i}{\hookrightarrow} \widetilde{\A}^k(\Omega) \overset{i}{\hookrightarrow} \A^k(\Omega)$ are bounded with norm $1$. The next result is now immediate.
\begin{lemma}\label{lemma5}
For a bounded $\Omega \subset \C^n$, the collection of seminorms $\{\normp{\cdot}_{k, \Omega} : k \in \N\}$ generates the same Fr\'{e}chet topology of $\Smooth$ as the collection $\{\norm{\cdot}_{k, \Omega} : k \in \N\}$, the $\fc^k$-seminorms of $\Omega$.
\end{lemma}
\begin{proof}
Let $k \in \N$. Note that for every $f \in \Smooth$,
\begin{align*}
 \normp{f}_{k, \Omega} &= \sup  \bigg \{\abs{D^{\alpha} f(z)} : z \in \Omega , \abs{\alpha}_{\infty} \leq k \bigg \}  \\
 &= \sup  \bigg \{\abs{D^{\alpha} f(z)} : z \in \Omega , \alpha_j \leq k \hspace{2mm} \text{for all $j = 1,2,\cdots,n$}\bigg \} \\
 & \leq \sup  \bigg \{\abs{D^{\alpha} f(z)} : z \in \Omega , [\alpha] \leq nk \bigg \} = \norm{f}_{nk, \Omega}.
\end{align*}
Also observe that for every $f \in \Smooth$,
\begin{align*}
 \norm{f}_{k, \Omega} &= \sup  \bigg \{\abs{D^{\alpha} f(z)} : z \in \Omega , [\alpha] \leq k \bigg \}  \\
 &\leq \sup  \bigg \{\abs{D^{\alpha} f(z)} : z \in \Omega , \alpha_j \leq k \hspace{2mm} \text{for all $j = 1,2,\cdots,n$}\bigg \} \\
 &= \sup  \bigg \{\abs{D^{\alpha} f(z)} : z \in \Omega , \abs{\alpha}_{\infty} \leq k \bigg \} = \normp{f}_{k, \Omega}.
\end{align*}
\end{proof}

\subsection{Absolute and unconditional convergence}\label{unconditionalconvergence}
A formal series $\sum_{\alpha \in \Gamma} x_{\alpha}$ in an LCTVS $X$, where $\Gamma$ is a countable index set, is said to be \emph{absolutely convergent} if there exists a bijection
\begin{equation}\label{eq:40D}
\sigma: \N \rightarrow \Gamma    
\end{equation}
such that for every continuous seminorm $p$ on $X$,
\begin{equation*}
\sum_{j=0}^{\infty} p (x_{\sigma(j)})    
\end{equation*}
is a convergent series of non-negative real numbers. Let $\mathsf{P}$ be a collection of continuous seminorms on $X$ that generates the locally convex topology of $X$. Then, to prove absolute convergence of the series $\sum_{\alpha \in \Gamma} x_{\alpha}$, it is sufficient to show that there exists a bijection $\sigma: \N \rightarrow \Gamma$ such that for every $p \in \mathsf{P}$,  the series $\sum_{j=0}^{\infty} p (x_{\sigma(j)}) < \infty$. A discussion of absolute and unconditional convergence for series in Banach spaces can be found in \cite{MR1442255}.\\
Let $(X, \tau)$ be an LCTVS and $(\Gamma, \geq)$ be a directed set. A net $(x_{\alpha})_{\alpha \in \Gamma}$ in $X$ is said to be a \emph{Cauchy net} if for every $\epsilon > 0$ and every continuous seminorm $p$ on $X$, there exists $\gamma \in \Gamma$ such that whenever $\alpha, \beta \in \Gamma$ and $\alpha, \beta \geq \gamma$, $p(x_{\alpha} - x_{\beta}) < \epsilon$. The net $(x_{\alpha})$ \emph{converges} to an element $x \in X$ if for every $\epsilon > 0$ and every continuous seminorm $p$ on $X$, there exists $\gamma \in \Gamma$ such that whenever $\alpha \geq \gamma$, we have $p(x_{\alpha} - x) < \epsilon$. Let $\mathsf{P}$ be a family of continuous seminorms on $X$ that generates the topology $\tau$. It follows from \eqref{eq:1B} that we can give an alternative definition of Cauchy net involving generating family of seminorms: a net $(x_{\alpha})_{\alpha \in \Gamma}$ in $X$ is said to be a \emph{Cauchy net} if for every $\epsilon > 0$ and every $p \in \mathsf{P}$, there exists $\gamma \in \Gamma$ such that whenever $\alpha, \beta \in \Gamma$ and $\alpha, \beta \geq \gamma$, $p(x_{\alpha} - x_{\beta}) < \epsilon$. An alternative definition of convergence of Cauchy nets involving generating seminorms can be given similarly. The space $X$ is said to be \emph{complete} if every Cauchy net of $X$ converges.
\begin{lemma}\label{lemma6A}
In a complete LCTVS, an absolutely convergent series is unconditionally convergent.
\end{lemma}
\begin{proof}
Let $\sum_{\alpha \in \Gamma} x_{\alpha}$ be an absolutely convergent series in a complete LCTVS $X$. So, there exists a bijection $\sigma: \N \rightarrow \Gamma$ such that for all continuous seminorm $p$ on $X$, the series $\sum_{j=0}^{\infty} p(x_{\sigma(j)})$ converges. Let $y_j = x_{\sigma(j)}$ and $s_k = \sum_{j=0}^k y_j$. Since $\sum_{j=0}^{\infty} p(y_j)$ converges, for $\epsilon > 0$ there exists $N_0 \in \N$ such that whenever $m, \ell \in \N$ with $m \geq \ell \geq N_0$, $\sum_{j = \ell+1}^{m} p(y_j) < \epsilon$. Therefore for $m \geq \ell \geq N_0$,
\begin{equation}\label{eq:60}
p(s_m - s_{\ell}) = p \Big (\sum_{j=\ell+1}^{m} y_j \Big ) \leq \sum_{j=\ell +1}^{m} p(y_j) < \epsilon.   \end{equation}
It follows from \eqref{eq:60} that the net $\{s_k\}$ is Cauchy in a complete LCTVS $X$, with directed set ($\N, \geq$), and therefore converges. Let $s_k \rightarrow s$ as $k \rightarrow \infty$. In order to complete the proof, it suffices to show that for every bijection $\tau: \N \rightarrow \N$, the series $\sum_{j=0}^{\infty} y_{\tau(j)}$ converges to the same limit $s$. Let $s_k^{\tau} = \sum_{j=0}^{k} y_{\tau(j)}$. We show $s_k^{\tau} \rightarrow s$ as $k \rightarrow \infty$. Choose $u \in \N$ such that the set of integers $\{0, 1,2,\cdots,N_0\}$ is contained in the set $\{\tau(0), \tau(1), \cdots, \tau(u)\}$. Then, if $k > u$, the elements $y_1, \cdots, y_{N_0}$ get cancelled in the difference $s_k - s_k^{\tau}$ and we have $p(s_k - s_k^{\tau}) < \epsilon$ by \eqref{eq:60}. This proves that the sequence $\{s_k\}$ and $\{s_k^{\tau}\}$ converges to the same sum. So, $s_k^{\tau} \rightarrow s$ as $k \rightarrow \infty$.
\end{proof}
\subsection{Convergence in the net of partial sums}\label{partialsums}
Let $(\mathfrak{F}(\Gamma), \subset)$ be the directed set of all finite subsets of a countable index set $\Gamma$ with inclusion `$\subset$' as its relation. Let $X$ be an LCTVS. The net $\mathfrak{F}(\Gamma) \rightarrow X$ defined by
\begin{equation}\label{eq:4A}
I \mapsto \sum_{\alpha \in I} x_{\alpha}    
\end{equation}
is said to be the \emph{net of partial sums} of the formal series $\sum_{\alpha \in \Gamma} x_{\alpha}$ in $X$. 
\begin{lemma}\label{lemma6AA}
In a complete LCTVS, the net of partial sums of an absolutely convergent series converges.
\end{lemma}
\begin{proof}
Let $\sum_{\alpha \in \Gamma} x_{\alpha}$ be an absolutely convergent series in a complete LCTVS $X$. So, there exists a bijection $\sigma: \N \rightarrow \Gamma$ such that for all continuous seminorm $p$ on $X$, the series of non-negative reals $ \sum_{j=0}^{\infty} p(x_{\sigma(j)})$ converges. Let $\epsilon > 0$ and let $p$ be a continuous seminorm on $X$. Since $ \sum_{j=0}^{\infty} p(x_{\sigma(j)})$ converges, there exists $N \in \N$ such that for all $m, k \in \N$ with $m \geq k > N_0$,
\begin{equation}\label{eq:5}
\sum_{j=k}^{m} p(x_{\sigma(j)}) < \epsilon/2.    
\end{equation}
Let $(\mathfrak{F}(\Gamma), \subset)$ be the directed set of all finite subsets of $\Gamma$ with inclusion as its order. Let $I = \{\sigma(0), \sigma(1),\cdots, \sigma(N_0)\}$, then $I \in \mathfrak{F}(\Gamma)$. Now, whenever $J, K \in \mathfrak{F}(\Gamma)$ with $J, K \supset I$,
\begin{align}
p \bigg ( \sum_{\alpha \in J} x_{\alpha} - \sum_{\alpha \in K} x_{\alpha}\bigg ) \leq p \bigg ( \sum_{\alpha \in J \setminus I} x_{\alpha} - \sum_{\alpha \in K \setminus I} x_{\alpha}\bigg ) &\leq p \bigg ( \sum_{\alpha \in J \setminus I} x_{\alpha} \bigg ) + p \bigg ( \sum_{\alpha \in K \setminus I} x_{\alpha} \bigg ) \nonumber \\
&\leq \sum_{\alpha \in J \setminus I} p(x_{\alpha}) + \sum_{\alpha \in K \setminus I} p(x_{\alpha}) \nonumber \\
& \leq \epsilon/2 + \epsilon/2 = \epsilon, \hspace{1mm} \text{from \eqref{eq:5}}.\nonumber
\end{align}
This shows that the net of partial sums $I \mapsto \sum_{\alpha \in I} x_{\alpha}$ of the series $\sum_{\alpha \in \Gamma} x_{\alpha}$ is Cauchy. Since $X$ is complete, it converges.
\end{proof}
\subsection{Covering a Reinhardt domain by polyannuli}
Let $\fp \subset \C^n$ be such that 
$$
\fp \coloneqq \big \{ z \in \C^n: r_j < \abs{z_j} < R_j \hspace{2mm} \text{or} \hspace{2mm} \abs{z_j}< R_j \quad \text{for $1\leq j \leq n$} \big \},
$$
where $0< r_j < R_j < \infty$ are reals. The set $\fp$ is said to be a \emph{polyannulus} in $\C^n$.
\begin{lemma}\label{lemma6}
Every Reinhardt domain in $\C^n$ is a countable union of polyannuli.
\end{lemma}
\begin{proof}
Let $\Omega \subset \C^n$ be Reinhardt domain let $a = (a_1, \cdots, a_n) \in \Omega$. Since $\Omega$ is open, there exists $\rho >0$ such that the ball $B_{\rho}(a)$ centered at $a$ with radius $\rho$ is contained in $\Omega$. Denote by $\abs{\Omega}$ the Reinhardt shadow of $\Omega$, that is $ \abs{\Omega} \coloneqq \big \{(\abs{z_1},\cdots,\abs{z_n}) : z \in \Omega \big \}$. Let $\varphi :\Omega \rightarrow \abs{\Omega} $ be the map defined by
\begin{equation}\label{eq:40}
\varphi(z) = (\abs{z_1},...,\abs{z_n}) \quad \text{for $z\in \Omega$}.    
\end{equation}
Since $\varphi$ is an open map, $\varphi(B_{\rho}(a))$ is an open set in $\abs{\Omega}$. So there exist rational numbers $0 < r_j < R_j < \infty$ such that $r_j < \abs{a_j} < R_j$ or $\abs{a_j} < R_j$ for every $1 \leq j \leq n$ and $\varphi(a) \in \prod_{j=1}^n A_j$ where 
\begin{equation*}
 A_j = 
 \begin{cases}
 (r_j, R_j) \hspace{2mm} \text{if $a_j \neq 0$} \\
 [0, R_j)  \hspace{2mm} \text{if $a_j = 0$}.
 \end{cases}
\end{equation*}
Therefore, $a \in \varphi^{-1}(\prod A_j)$, which is a polyannulus
$$
\fp_a = \big \{z \in \Omega: r_j < \abs{z_j} < R_j \hspace{2mm} \text{or} \hspace{2mm} \abs{z_j}< R_j \quad \text{for $1\leq j \leq n$} \big \}.
$$
Consequently $\Omega \subset \bigcup \fp_a$, and the union is a countable one.
\end{proof}

\section{Absolute convergence of Laurent series for a bounded domain }\label{absoluteconvergence}
For an integer $N \geq 0$, let $Q_N = \{\alpha \in \Z^n : \abs{\alpha}_{\infty} \leq N\}$ denote the \enquote{box} of edge $2N$, so that $Q_N$ has $(2N+1)^n$ points. Construct a bijection $\sigma: \N \rightarrow \Z^n$ by successively enumerating $Q_0, Q_1 \smallsetminus Q_0, Q_2 \smallsetminus Q_1$ and so on. Therefore $\sigma$ satisfies
\begin{equation}\label{eq:40A}
  \begin{cases}
                                   \sigma(0) = 0, \\
                                   \sigma(1), \cdots, \sigma(3^n) \hspace{1.5mm} \text{is an enumeration of $Q_1 \smallsetminus Q_0$}, \\
                                    \text{in general} \hspace{1.5mm} \bigg \{\sigma(j) : (2N-1)^n +1 \leq j \leq (2N+1)^n \bigg \} = Q_N \smallsetminus Q_{N-1}.
 
  \end{cases}
  \end{equation}
Let $M \geq 1$ be an integer and let $M_1$ be the largest integer such that $(2M_1+1)^n \leq M$, so that $M_1 = \floor{\frac{\sqrt[\leftroot{-2}\uproot{2} n]{M} - 1}{2}}$, where $\floor{\cdot}$ is the floor function. Then it is clear that
\begin{equation}\label{eq:40B}
Q_{M_1} \subset \bigg \{\sigma(0), \sigma(1), \cdots , \sigma(M) \bigg \} \subset Q_{M_1 +1}. 
\end{equation}
Let 
\begin{equation}\label{eq:40C}
f = \sum_{\alpha \in \Z^n} c_{\alpha} e_{\alpha}    
\end{equation} 
be the Laurent series expansion of $f \in \Smooth$, where $e_{\alpha}$ denotes the \emph{Laurent monomial} of exponent $\alpha: e_{\alpha}(z) = z^{\alpha} = z_1^{\alpha_1} \cdots z_n^{\alpha_n}$. In this section our goal is to prove the following result.
\begin{theorem}\label{theorem2}
Let $\Omega$ be a bounded Reinhardt domain in $\C^n$ and $f \in \Smooth$. Then the Laurent series of $f$ in \eqref{eq:40C} converges absolutely in the topology of $\Smooth$. In fact, the bijection $\sigma$ in \eqref{eq:40D} can be taken to be the one constructed in \eqref{eq:40A} above.
\end{theorem}
Recall from Section \ref{gensem} that if $\Omega$ is bounded, the collection of seminorms $\{\normp{\cdot}_{k, \Omega} : k \in \N\}$ generates the Fr\'{e}chet topology of $\Smooth$, where for each $k$, the seminorm $\normp{\cdot}_{k, \Omega}$ is as in \eqref{eq:24}. To prove Theorem \ref{theorem2}, we show that if $\sigma$ is taken as in \eqref{eq:40A} above, then for every $k \in \N$, $\sum_{j=0}^{\infty} \normp{c_{\sigma(j)}e_{\sigma(j)}}_{k, \Omega} < \infty$.\\
The following proposition is the key to prove Theorem \ref{theorem2}.
\begin{proposition}\label{proposition6}
Let $\fp$ be a polyannulus in $\C^n$, that is 
$$
\fp \coloneqq \big \{ z \in \C^n: r_j < \abs{z_j} < R_j \hspace{2mm} \text{or} \hspace{2mm} \abs{z_j}< R_j \quad \text{for $1\leq j \leq n$} \big \},
$$
where $0 < r_j < R_j < \infty$ are reals. For an integer $\ell$, consider 
\begin{equation}\label{eq:50}
    \mu_{\ell} =
  \begin{cases}
                                    \dfrac{1}{\ell (\ell -1)} & \text{if $\ell \neq 0,1$} \\
                                   1 & \text{if $\ell = 0,1$}. \\
 
  \end{cases}
  \end{equation}
Let $\alpha = (\alpha_1,\cdots,\alpha_n) \in \Z^n$, $k \in \N$ and $M_{\alpha, k} = \prod_{j=1}^n \mu_{\alpha_j- k}$. Suppose $f = \sum_{\gamma \in \Z^n} c_{\gamma} e_{\gamma}$ is the Laurent series expansion of $f \in \poly$, where $e_{\gamma}$ is the monomial function of exponent $\gamma$. Then
\begin{equation}\label{eq:56}
\normp{c_{\alpha} e_{\alpha}}_{k, \fp} \leq  \bigg ( M_{\alpha, k} \cdot  \prod_{j=1}^n 
(1+ R_j^2) \bigg ) \cdot \normp{f}_{k+2, \fp}.   
\end{equation}
where $R_j$'s are as in the statement and  $\normp{\cdot}_{k, \fp}$ is as in \eqref{eq:24}.
\end{proposition}

\begin{proof}
Let $Z = \{z \in \C^n : z_j = 0 \hspace{2mm}\text{for some $j$}\}$ and $z \in \fp \setminus Z$. One can write the coefficient $c_{\alpha} = c_{\alpha}(f)$ of the Laurent series of $f$ using Cauchy formula:
\begin{equation}\label{eq:48}
c_{\alpha} =  \frac{1}{(2 \pi i)^n} \int_{\abs{\zeta_1} = \abs{z_1}} \cdots \int_{\abs{\zeta_n} = \abs{z_n}} \frac{f(\zeta)}{\zeta^{\alpha}} \hspace{1mm} \frac{d\zeta_n}{\zeta_n} \cdots \frac{d\zeta_1}{\zeta_1}
\end{equation}
Note that if $A_j$ is a disc for some $j$, then $c_{\alpha} = 0$ whenever $\alpha_j < 0$. Fix some multi-index notations: $z^{\alpha} = z_1^{\alpha_1} \cdots z_n^{\alpha_n}$ and $z \cdot e^{i \theta} = (z_1 e^{i \theta_1}, \cdots, z_n e^{i \theta_n})$ and use \enquote{vector-like} notation: $\langle \alpha,\theta \rangle = \alpha_1 \theta_1 + \cdots + \alpha_n \theta_n$. Parametrize the contours in \eqref{eq:48} by $\zeta_j = z_j e^{i \theta_j}$ for every $j$ to get 
\begin{equation}\label{eq:48A}
c_{\alpha} =  \frac{1}{(2 \pi)^n} \int_{0}^{2 \pi} \cdots \int_{0}^{2 \pi} \frac{f(z \cdot e^{i\theta})}{z^{\alpha} \exp(i\langle \alpha,\theta \rangle)}\hspace{2mm} d\theta_n \cdots d\theta_1. 
\end{equation}
Consider the case $k =0$. Introduce the multi-index $\beta \in \Z^n$ (depending on $\alpha$) as 
\[
\beta_j = 
\begin{cases}
                     2 & \text{if $\alpha_j \neq 0,1$} \\
                     
                     0 & \text{if $\alpha_j = 0,1$}. \\
\end{cases}
\]
We claim that
\begin{equation}\label{eq:65}
 c_{\alpha} z^{\alpha} = \frac{z^{\beta}}{(2 \pi)^n} \int_{0}^{2 \pi} \cdots \int_{0}^{2 \pi}  \bigg ( \prod_{j=1}^n U(\alpha_j, \theta_j) \bigg ) \frac{\partial^{\abs{\beta}} f}{\partial \zeta_1^{\beta_1}\cdots \partial \zeta_n^{\beta_n}}(z \cdot e^{i \theta}) \hspace{2mm} d\theta_n \cdots d\theta_1,  
\end{equation}
where, for each $1 \leq j \leq n$
\[
  U(\alpha_j, \theta_j) =
  \begin{cases}
                                    \dfrac{e^{-i (\alpha_j -2) \theta_j}}{\alpha_j (\alpha_j -1)} & \text{if $\alpha_j \neq 0,1$} \\
                                   e^{-i \alpha_j \theta_j} & \text{if $\alpha_j = 0,1$}. \\
 
  \end{cases}
\]
To prove \eqref{eq:65}, let us write \eqref{eq:48A} as
\begin{equation}\label{eq:66}
c_{\alpha} z^{\alpha} = \frac{1}{(2 \pi)^n} \int \limits_{ 0}^{2 \pi} \cdots \bigg ( \int \limits_{0}^{2 \pi}   \frac{f(z_1 e^{i \theta_1},\cdots,z_n e^{i \theta_n})}{\exp(i \alpha_n \theta_n)} d\theta_n \bigg ) \frac{d\theta_{n-1}\cdots d\theta_1}{\exp(i(\alpha_1 \theta_1 + \cdots + \alpha_{n-1} \theta_{n-1}))}.  \end{equation}
Let $I_n$ be the integral inside the parentheses in \eqref{eq:66}. That is 
\begin{equation}\label{eq:66D}
I_n =  \int_{0}^{2 \pi} e^{-i\alpha_n \theta_n} \hspace{1mm} f(z \cdot e^{i \theta}) \hspace{1mm} d\theta_n.   
\end{equation}
If $\alpha_n = 0,1$, we set $\beta_n = 0$ and one can write an expression for $I_n$ (in terms of $\beta_n$) from \eqref{eq:66D} as, 
\begin{equation}\label{eq:66a}
I_n =  z_n^{\beta_n} \int_{0}^{2 \pi} e^{-i\alpha_n \theta_n} \hspace{1mm} \frac{\partial^{\beta_n} f}{\partial \zeta_n^{\beta_n}}(z \cdot e^{i \theta}) \hspace{1mm} d\theta_n.  \end{equation}
If $\alpha_n \neq 0,1$, we integrate $I_n$ in \eqref{eq:66D} by parts with respect to $\theta_n$ as follows: take $u= u(\theta_n) = f(z \cdot e^{i \theta})$ and $dv = e^{-i \alpha_n \theta_n} \hspace{1mm}d\theta_n$ in the formula $\int_{0}^{2 \pi} u dv = [uv]_{0}^{2 \pi} - \int_{0}^{2 \pi} v du$ and note that the first term vanishes due to periodicity. We get
\begin{equation}\label{eq:66A}
I_n = - \int_{ 0}^{2 \pi} \frac{e^{-i \alpha_n \theta_n}}{- i \alpha_n} \hspace{2mm}\frac{\partial}{\partial \theta_n}f (z \cdot e^{i \theta})  \hspace{2mm} d\theta_n.
\end{equation}
We use the chain rule: $\frac{\partial}{\partial \theta_n}f (z \cdot e^{i \theta}) = \frac{\partial f}{\partial \zeta_n}(z \cdot e^{i \theta}) \cdot z_n i e^{i \theta_n}$. After simplifying we get 
\begin{equation}\label{eq:66C}
I_n = z_n \int_{0}^{2 \pi} \frac{e^{-i(\alpha_n -1)\theta_n}}{\alpha_n} \hspace{2mm}\frac{\partial f}{\partial \zeta_n}(z \cdot e^{i \theta})  \hspace{2mm} d\theta_n.    
\end{equation}
Using integration by parts again in the same way we have
\begin{equation}\label{eq:66B}
I_n = z_n^2 \int_{0}^{2 \pi} \frac{e^{-i(\alpha_n -2)\theta_n}}{\alpha_n(\alpha_n - 1)} \hspace{2mm}\frac{\partial^2 f}{\partial \zeta_n^2}(z \cdot e^{i \theta})  \hspace{2mm} d\theta_n.    
\end{equation}
Recall that we set $\beta_n = 2$ for $\alpha_n \neq 0,1$. Therefore \eqref{eq:66B} can be rewritten (in terms of $\beta_n$) as
\begin{equation}\label{eq:67}
I_n = z_n^{\beta_n} \int_{0}^{2 \pi} \frac{e^{-i(\alpha_n -2)\theta_n}}{\alpha_n(\alpha_n-1)} \hspace{1mm}\frac{\partial^{\beta_n} f}{\partial \zeta_n^{\beta_n}}(z \cdot e^{i \theta}) \hspace{1mm} d\theta_n. \end{equation}
We substitute the expressions \eqref{eq:66a} or \eqref{eq:67} for $I_n$ in \eqref{eq:66} (depending on the values of $\alpha_n$, and therefore $\beta_n$). Rearranging the terms and the integrals we write \eqref{eq:66} as
\begin{equation}\label{eq:67A}
 c_{\alpha}z^{\alpha}= \frac{ z_n^{\beta_n}}{(2 \pi)^n} \int \limits_{\theta_n} U(\alpha_n, \theta_n)  \int \limits_{\theta_1} \cdots \bigg ( \int \limits_{\theta_{n-1}} \frac{\frac{\partial^{\beta_n} f}{\partial \zeta_n^{\beta_n}}(z \cdot e^{i \theta})}{\exp(i \alpha_{n-1} \theta_{n-1})} d\theta_{n-1} \bigg ) \frac{d\theta_{n-2}\cdots d\theta_1}{\exp(i(\alpha_1 \theta_1 + \cdots +\alpha_{n-2} \theta_{n-2}))} d\theta_n
\end{equation}
Let $I_{n-1}$ be the integral inside the parenthesis in \eqref{eq:67A}. That is 
\begin{equation}\label{eq:67B}
I_{n-1} = \int_{0}^{2 \pi} e^{-i\alpha_{n-1} \theta_{n-1}} \hspace{1mm} \frac{\partial^{\beta_n} f}{\partial \zeta_n^{\beta_n}}(z \cdot e^{i \theta}) \hspace{1mm} d\theta_{n-1}.    
\end{equation}   
Depending on the values of $\alpha_{n-1}$ we repeat the earlier procedure. If $\alpha_{n-1} = 0,1$, we write the expression of $I_{n-1}$ in terms of $\beta_{n-1}$ from \eqref{eq:67B}, otherwise we integrate $I_{n-1}$ in \eqref{eq:67B} by parts with respect to the variable $\theta_{n-1}$ as previous. We substitute the expressions for $I_{n-1}$ in \eqref{eq:67A} and so on. To prove that our claim \eqref{eq:65} is true, we repeat the same procedure $(n-2)$ more times.\\
Now take absolute values on both sides in \eqref{eq:65}. We get
\begin{align}
\abs{c_{\alpha} z^{\alpha}} &= \abs{\frac{z^{\beta}}{(2 \pi)^n} \int_{0}^{2 \pi} \cdots \int_{0}^{2 \pi}  \bigg ( \prod_{j=1}^n U(\alpha_j, \theta_j) \bigg ) \frac{\partial^{\abs{\beta}} f}{\partial \zeta_1^{\beta_1}\cdots \partial \zeta_n^{\beta_n}}(z \cdot e^{i \theta}) \hspace{2mm} d\theta_n \cdots d\theta_1} \nonumber \\
&\leq R^{\beta} \cdot \prod_{j=1}^n \abs{U(\alpha_j , \theta_j)} \cdot  \norm{\frac{\partial^{\abs{\beta}} f}{\partial \zeta_1^{\beta_1}\cdots \partial \zeta_n^{\beta_n}}}_{T} \leq \bigg ( R^{\beta} \prod_{j=1}^n \abs{U(\alpha_j , \theta_j)} \bigg ) \cdot \normp{f}_{2, \fp} \label{eq:101}
\end{align}
where $R^{\beta} = R_1^{\beta_1}\cdots R_n^{\beta_n}$, $T = \{\abs{\zeta_j} = \abs{z_j} : 1 \leq j \leq n\}$ is a torus contained in $\fp$ and $\norm{\cdot}_{T}$ denotes the sup norm on $T$. Note that $R^{\beta} \leq \prod_{j=1}^n (1+ R_j^2)$.
So, it follows from \eqref{eq:101} that for $z \in \fp \setminus Z$,
\begin{equation}\label{eq:103B}
\abs{c_{\alpha} z^{\alpha}} \leq \bigg (\prod_{j=1}^n  \abs{U(\alpha_j, \theta_j)}\bigg ) \cdot  \prod_{j=1}^n (1+ R_j^2) \cdot \normp{f}_{2, \fp} = \bigg ( M_{\alpha, 0} \cdot   \prod_{j=1}^n (1+ R_j^2) \bigg ) \cdot \normp{f}_{2, \fp}
\end{equation}
where we note from \eqref{eq:50} that for each $1 \leq j \leq n$, $\abs{U(\alpha_j, \theta_j)} = \mu_{\alpha_j}$ and therefore $M_{\alpha, 0} = \prod_{j=1}^n \mu_{\alpha_j} = \prod_{j=1}^n \abs{U(\alpha_j, \theta_j)}$.\\
In \eqref{eq:103B} we have proved that the function $z \mapsto c_{\alpha}z^{\alpha}$ is bounded on the set $\fp \setminus Z$. By the Riemann removable singularity theorem (see \cite[p. 32]{MR847923}) we can extend the function holomorphically to $\fp$ and the extended function admits the same bound. This proves our result for the case $k=0$. \\
Let $k \geq 1$ and $\gamma \in \N^n$.  We multiply by $z^{\alpha}$ and apply $D^{\gamma}$ in both sides of \eqref{eq:48A} to get,
\begin{align}
D^{\gamma}(c_{\alpha}(f) z^{\alpha}) &= D^{\gamma} \bigg ( \frac{1}{(2 \pi)^n} \int_{0}^{2 \pi} \cdots \int_{0}^{2 \pi} \frac{f(z \cdot e^{i\theta})}{ \exp(i \langle \alpha, \theta \rangle)}\hspace{2mm} d\theta_n \cdots d\theta_1 \bigg ) \nonumber \\
&= \frac{1}{(2 \pi)^n} \int_{0}^{2 \pi} \cdots \int_{0}^{2 \pi} D^{\gamma} \big ( f (z \cdot e^{i \theta}) \big )\hspace{1mm} e^{-i\langle \alpha, \theta \rangle} \hspace{2mm} d\theta_n \cdots d\theta_1 \nonumber \\
&= \frac{1}{(2 \pi)^n} \int_{0}^{2 \pi} \cdots \int_{0}^{2 \pi} (D^{\gamma} f)(z \cdot e^{i \theta}) \hspace{1mm} e^{i\langle \gamma, \theta \rangle} \hspace{1mm} e^{-i\langle \alpha, \theta \rangle} \hspace{2mm} d\theta_n \cdots d\theta_1 \nonumber \\
&= \frac{1}{(2 \pi)^n} \int_{0}^{2 \pi} \cdots \int_{0}^{2 \pi} (D^{\gamma} f)(z \cdot e^{i \theta}) \hspace{1mm} e^{-i\langle \alpha-\gamma, \theta \rangle} \hspace{2mm} d\theta_n \cdots d\theta_1 \nonumber \\
&= c_{\alpha - \gamma} (D^{\gamma} f) z^{\alpha - \gamma}, \qquad \text{by \eqref{eq:48A}}.
\end{align}
Now, we use the boundedness for the case $k=0$ and the fact that $D^{\gamma}f \in \A^{\infty}(\fp)$ to get,
\begin{equation}\label{eq:104}
\abs{D^{\gamma}(c_{\alpha}(f) z^{\alpha})} = \abs{c_{\alpha - \gamma} (D^{\gamma} f) z^{\alpha - \gamma}} \leq \bigg ( M_{\alpha-\gamma, 0} \hspace{2mm} \prod_{j=1}^n (1 + R_j^2) \bigg ) \cdot \normp{D^{\gamma}f}_{2, \fp}.   
\end{equation}
Let $\gamma$ be such that $\abs{\gamma}_{\infty} \leq k$, that is $\gamma_j \leq k$ for every $1 \leq j \leq n$. So, $\frac{1}{\alpha_j - \gamma_j} \leq \frac{1}{\alpha_j - k}$ and therefore  $\mu_{\alpha_j - \gamma_j} \leq \mu_{\alpha_j - k}$. Now it follows that 
\begin{equation}\label{eq:51}
M_{\alpha-\gamma, 0} = \prod_{j=1}^n \mu_{\alpha_j - \gamma_j} \leq \prod_{j=1}^n \mu_{\alpha_j - k} = M_{\alpha, k}.    
\end{equation}
Moreover,
\begin{align}
\normp{D^{\gamma} f}_{2, \fp} = \sup_{z \in \fp} \bigg \{ \abs{D^{\alpha}(D^{\gamma}f) (z)} : \abs{\alpha}_{\infty} \leq 2\bigg \}
&=  \sup_{z \in \fp} \bigg \{ \abs{D^{\beta}f (z)} : \beta_j \leq \gamma_j + 2 \hspace{2mm} \text{for all $j$}\bigg \} \nonumber \\ 
&=  \sup_{z \in \fp} \bigg \{ \abs{D^{\beta}f (z)} : \abs{\beta}_{\infty} \leq \abs{\gamma}_{\infty} + 2 \bigg \} \nonumber \\ 
&= \normp{f}_{\abs{\gamma}_{\infty} + 2, \fp} \leq \normp{f}_{k+2, \fp}. \label{eq:52}
\end{align}
Therefore from \eqref{eq:104}, \eqref{eq:51} and \eqref{eq:52} we get
\begin{equation}\label{eq:105}
 \sup_{z \in \fp} \abs{D^{\gamma}(c_{\alpha}(f) z^{\alpha})} \leq \bigg ( M_{\alpha, k} \cdot  \prod_{j=1}^n (1+ R_j^2) \bigg ) \cdot \normp{f}_{k+2, \fp} 
\end{equation}
Taking supremum in the left over all $\gamma$ such that $\abs{\gamma}_{\infty} \leq k$ we get the result.
\end{proof}

\begin{proof}[Proof of Theorem \ref{theorem2}]
Let $z \in \Omega$. By Lemma \ref{lemma6}, there exists a polyannulus $\fp \subset \Omega$ such that $z \in \fp$. So, for every $1 \leq j \leq n$, either $\big \{r_j < \abs{z_j} < R_j \big\}$ or $\big \{\abs{z_j} < R_j \big \}$ and $0 < r_j < R_j < \infty$ are rational numbers. Let $f = \sum_{\gamma \in \Z^n} c_{\gamma} e_{\gamma}$ be the Laurent series expansion of a function $f  \in \A^{\infty}(\Omega) \subset \poly$. Let $k \in \N$ and $\alpha \in \Z^n$. Let $ M_{\alpha, k}$ and $\mu_{\alpha_j - k}$ be as in the statement of Proposition \ref{proposition6}. Now, by Proposition \ref{proposition6} we can write
\begin{equation}\label{eq:75}
\normp{c_{\alpha} e_{\alpha}}_{k, \fp} \leq  C_{\alpha, \fp} \cdot \normp{f}_{k+2, \fp}.   
\end{equation}
where the constant $C_{\alpha, \fp} = M_{\alpha, k} \cdot \prod_{j=1}^n (1+ R_j^2) $. Since $\fp \subset \Omega$, $\normp{f}_{k+2, \fp} \leq \normp{f}_{k+2, \Omega}$. Since $\Omega$ is bounded, there exists a finite number $B$ (independent of the polyannulus $\fp$, as long as $\fp \subset \Omega$) such that $\prod_{j=1}^n (1+ R_j^2) \leq B$. Therefore we can write
\begin{equation}\label{eq:78}
\normp{c_{\alpha} e_{\alpha}}_{k, \fp} \leq  C_{\alpha} \cdot \normp{f}_{k+2, \Omega}.   
\end{equation}
 where the constant $C_\alpha = M_{\alpha, k} \cdot B$. We take supremum in the left of \eqref{eq:78} over all $\fp$'s contained in $\Omega$ to get 
\begin{equation}\label{eq:80}
\normp{c_{\alpha} e_{\alpha}}_{k, \Omega} \leq  C_{\alpha} \cdot \normp{f}_{k+2, \Omega}.   
\end{equation}
Recall $\abs{\alpha}_{\infty} \coloneqq \max \big \{ \abs{\alpha_j}, 1 \leq j \leq n \big \}$, where $\alpha \in \Z^n$. Now, for every $N \in \N$
\begin{equation}
\sum_{\abs{\alpha}_{\infty} \leq N} C_{\alpha} = B \cdot \prod_{j=1}^n \big ( \sum_{\alpha_j = -N}^{N} \mu_{\alpha_j-k} \big ).   
\end{equation}
For every $j$, we write $\lim \limits_{N \rightarrow \infty} \sum \limits_{\alpha_j = -N}^{N} \mu_{\alpha_j-k} = \lim \limits_{N \rightarrow \infty} \sum \limits_{\alpha_j = -N}^{-1} \mu_{\alpha_j-k} + \lim \limits_{N \rightarrow \infty} \sum \limits_{\alpha_j = 0}^{N} \mu_{\alpha_j-k}$ and observe that both the terms are finite (using the limit comparison test with the series $\sum_{j=2}^{\infty} \frac{1}{j(j-1)}$). So, $\lim \limits_{N \rightarrow \infty} \sum \limits_{\alpha_j = -N}^{N} \mu_{\alpha_j-k} < \infty$ and therefore 
\begin{equation}\label{eq:79}
\lim_{N \rightarrow \infty} \sum_{\abs{\alpha}_{\infty} \leq N} C_{\alpha} = B \cdot \lim_{N \rightarrow \infty} \prod_{j=1}^n \big ( \sum_{\alpha_j = -N}^{N} \mu_{\alpha_j-k} \big ) < \infty.  
\end{equation}
Let $m \geq 1$ be an integer and let $\sigma: \N \rightarrow \Z^n$ be the bijection constructed in \eqref{eq:40A}. Then it follows from \eqref{eq:40B} that 
\begin{equation}\label{eq:81}
\sum_{\abs{\alpha}_{\infty} \leq \floor{\frac{\sqrt[\leftroot{-2}\uproot{2} n]{m} - 1}{2}}} \normp{c_{\alpha} e_{\alpha}}_{k, \Omega} < \sum_{j=0}^{m}\normp{c_{\sigma(j)} e_{\sigma(j)}}_{k, \Omega} < \sum_{\abs{\alpha}_{\infty} \leq \floor{\frac{\sqrt[\leftroot{-2}\uproot{2} n]{m} - 1}{2}} +1} \normp{c_{\alpha} e_{\alpha}}_{k, \Omega}.   \end{equation}
Now, from \eqref{eq:80} and \eqref{eq:79},
\begin{equation}\label{eq:82}
 \lim_{m \rightarrow \infty}  \sum_{\abs{\alpha}_{\infty} \leq \floor{\frac{\sqrt[\leftroot{-2}\uproot{2} n]{m} - 1}{2}}} \normp{c_{\alpha} e_{\alpha}}_{k, \Omega} < \infty. 
\end{equation}
Therefore from \eqref{eq:81},
\[ \sum_{j=0}^{\infty}\normp{c_{\sigma(j)} e_{\sigma(j)}}_{k, \Omega} = \lim_{m \rightarrow \infty} \sum_{j=0}^{m}\normp{c_{\sigma(j)} e_{\sigma(j)}}_{k, \Omega} < \infty.\]
The absolute convergence follows from here.
\end{proof}
\section{Proof of main theorem}\label{mainresult}
Let $f = \sum_{\alpha \in \Z^n} c_{\alpha} e_{\alpha}$ be the Laurent series expansion of $f \in \Smooth$. Recall from Section \ref{gensem} that one can write $\Omega =  \bigcup_{m \in \N} \Omega_m$, where $\Omega_m$ is bounded for each $m \in \N$. Let $\sigma$ be the bijection constructed in \eqref{eq:40A}. By Theorem \ref{theorem2}, for all integers $m, k \geq 0$, $\sum_{j=0}^{\infty} \normp{c_{\sigma(j)}e_{\sigma(j)}}_{k,\Omega_m} < \infty$, where $\normp{\cdot}_{k, \Omega_m}$ is as in \eqref{eq:24}. Also recall the seminorms $\norm{\cdot}_{k,\Omega_m}$ as in \eqref{eq:1A}, where $k \in \N$ and $\Omega_m$ is bounded. Since $\norm{c_{\sigma(j)}e_{\sigma(j)}}_{k, \Omega_m} \leq \normp{c_{\sigma(j)}e_{\sigma(j)}}_{k, \Omega_m}$, it follows easily that for all integers $m, k \geq 0$, \[\sum_{j=0}^{\infty} \norm{c_{\sigma(j)}e_{\sigma(j)}}_{k,\Omega_m} < \infty.\] But, we saw in Section \ref{gensem} that the collection of seminorms $\big \{ \norm{\cdot}_{k, \Omega_m} : k, m \in \N \big \}$ generates the Fr\'{e}chet topology of $\Smooth$. This proves that the Laurent series of $f$ converges absolutely in $\Smooth$, therefore unconditionally in $\Smooth$ (see Lemma \ref{lemma6A}). So, the series $\sum_{j=0}^{\infty} c_{\sigma(j)} e_{\sigma(j)}$ converges in $\Smooth$. Let 
\begin{equation}\label{eq:82AAA}
g = \lim_{N \rightarrow \infty} \sum_{j=0}^{N} c_{\sigma(j)} e_{\sigma(j)}     
\end{equation}
in $\Smooth$. Since the inclusion map $\Smooth \overset{i}{\hookrightarrow} \fo(\Omega)$ is continuous, the Equation \eqref{eq:82AAA} holds in $\fo(\Omega)$ as well. But, the Laurent series of a holomorphic function $f$ on $\Omega$ converges to $f$ absolutely and uniformly on compact subsets of $\Omega$ (see \cite[p. 46]{MR847923}). Therefore $f = g$, and the proof is complete.

\bibliographystyle{abbrv}
\bibliography{ref}


\end{document}